\documentclass[12pt]{amsart}
\usepackage{amscd,amsmath,amsthm,amssymb}
\usepackage{color}

\def\NZQ{\mathbb} 
\def\bbN{{\NZQ N}}              
\def\NN{{\NZQ N}}
\def\bbM{{\NZQ M}}

\def\ZZ{{\NZQ Z}}

\def\EE{{\NZQ E}}

\def\frk{\mathfrak}               
\def\aa{{\frk a}}

\def\mm{{\frk m}}
\def\m{{\frk m}}

\def\nn{{\frk n}}

\def\ab{{\bold a}}
\def\bb{{\bold b}}
\def\cb{{\bold c}}

\def\al{\alpha}
\def\bl{\beta}

\def\opn#1#2{\def#1{\operatorname{#2}}} 
\opn\chara{char} \opn\length{\ell} \opn\pd{pd} \opn\rk{rk}
\opn\projdim{proj\,dim} \opn\injdim{inj\,dim} \opn\rank{rank}
\opn\depth{depth} \opn\grade{grade} \opn\height{height}
\opn\Det{Det}\opn\size{size}
\opn\embdim{emb\,dim} \opn\codim{codim}

\opn\Tr{Tr} \opn\bigrank{big\,rank}
\opn\superheight{superheight}\opn\lcm{lcm}
\opn\trdeg{tr\,deg}
\opn\reg{reg} \opn\lreg{lreg} \opn\ini{in} \opn\lpd{lpd}
\opn\size{size}\opn{\mult}{mult}
\opn{\Cl}{Cl}
\opn{\PF}{PF}
\opn{\RF}{RF}
\opn{\MC}{MC}
\opn{\Kos}{Kos}
\opn{\rmH}{H}
\opn\div{div} \opn\Div{Div} \opn\cl{cl} \opn\Cl{Cl}
\opn\Spec{Spec} \opn\Supp{Supp} \opn\supp{supp} \opn\Sing{Sing}
\opn\Ass{Ass} \opn\Min{Min} \opn\cl{cl} 
\opn\Ann{Ann} \opn\Rad{Rad} \opn\Soc{Soc}
\opn\Syz{Syz} \opn\Im{Im} \opn\Ker{Ker} \opn\Coker{Coker}
\opn\Am{Am} \opn\Hom{Hom} \opn\Tor{Tor} \opn\Ext{Ext}
\opn\End{End} \opn\Aut{Aut} \opn\id{id} \opn\ini{in}\opn\GCD{GCD}
\opn\nat{nat} \opn\Soc{Soc}
\opn\emb{emb}
\opn\pff{pf}
\opn\Pf{Pf} \opn\GL{GL} \opn\SL{SL} \opn\mod{mod} \opn\ord{ord}
\opn\Gin{Gin}
\opn\Hilb{Hilb}\opn\adeg{adeg}\opn\std{std}\opn\ip{infpt}
\opn\Pol{Pol}
\opn\sat{sat}
\opn\Var{Var}
\opn\Gen{Gen}
\opn\lex{lex}
\opn\div{div}
\opn\NUF{NUF}
\opn\mNUF{mNUF}
\opn\type{type}

\opn\PF{PF}
\opn\Fr{F}
\opn\Ap{Ap}
\opn\Deg{Deg}
\opn\aff{aff} \opn\con{conv} \opn\relint{relint} \opn\st{st}
\opn\lk{lk} \opn\cn{cn} \opn\core{core} \opn\vol{vol}
\opn\link{link} \opn\star{star}
\opn\gr{gr}

\def\pot#1#2{#1[\kern-0.28ex[#2]\kern-0.28ex]}

\opn\dirlim{\underrightarrow{\lim}}
\opn\inivlim{\underleftarrow{\lim}}
\let\to=\rightarrow

\def\Implies{\ifmmode\Longrightarrow \else
        \unskip${}\Longrightarrow{}$\ignorespaces\fi}
\def\implies{\ifmmode\Rightarrow \else
        \unskip${}\Rightarrow{}$\ignorespaces\fi}
\def\iff{\ifmmode\Longleftrightarrow \else
        \unskip${}\Longleftrightarrow{}$\ignorespaces\fi}

\let\:=\colon
\newtheorem{Theorem}{Theorem}[section]
\newtheorem{thm}{Theorem}[section]

\newtheorem{lem}[Theorem]{Lemma}

\newtheorem{cor}[Theorem]{Corollary}

\newtheorem{prop}[Theorem]{Proposition}

\newtheorem{rem}[Theorem]{Remark}

\newtheorem{ex}[Theorem]{Example}

\newtheorem{defn}[Theorem]{Definition}

\newtheorem*{acknowledgement}{Acknowledgement}

\let\epsilon\varepsilon
\let\kappa=\varkappa
\def\qed{\ifhmode\textqed\fi
      \ifmmode\ifinner\quad\qedsymbol\else\dispqed\fi\fi}
\def\textqed{\unskip\nobreak\penalty50
       \hskip2em\hbox{}\nobreak\hfil\qedsymbol
       \parfillskip=0pt \finalhyphendemerits=0}
\def\dispqed{\rlap{\qquad\qedsymbol}}

\opn\dis{dis}
\def\pnt{{\raise0.5mm\hbox{\large\bf.}}}

\opn\Lex{Lex}
\opn\int{int}

\newcommand{\inD}[1][\relax]{\def\argone{#1}\def\temprelax{\relax}
  \ifx\argone\temprelax\right.\else\,\middle|#1\right.{}\fi}

\newif\ifbinary
\binarytrue
\newcommand{\Z}{\mathbb Z}

\begin{document}

\title{Inverse polynomials of numerical semigroup rings}
\author{Kazufumi Eto}
\address[Kazufumi Eto]{Department of Mathematics, 
Nippon Institute of Technology, Miyashiro, Saitama, Japan 345-8501}
\email{etou@nit.ac.jp}

\author{Kei-ichi Watanabe}
\address[Kei-ichi Watanabe]{Department of Mathematics, College of Humanities and Sciences, 
Nihon University, Setagaya-ku, Tokyo, 156-8550, Japan and 
Organization for the Strategic Coordination of Research and Intellectual Properties, Meiji University
}
\email{watanabe@math.chs.nihon-u.ac.jp}

\thanks{The second author was partially supported by JSPS  KAKENHI 
Grant Number  20K03522}
\subjclass[2020]{Primary 13H10; Secondary 20M14, 13D02, 14M10}
\begin{abstract}
We study the defining ideal of a numerical semigroup ring $k[H]$ 
 using the inverse polynomial attached to the Artinian ring $k[H]/(t^h)$.
We give a criterion of $H$ to be symmetric or almost symmetric 
using the annihilator of the inverse system. Also we give characterization of 
symmetric numerical semigroups with small multiplicity and give a new    
proof of Bresinsky's Theorem for symmetric semigroups generated by $4$ 
elements. 
\end{abstract}

\maketitle

\section*{Introduction}

Numerical semigroups and semigroup rings are very important objects in the study of singularities of
dimension $1$ (See \S 1 for the definitions).  Let $H =\langle n_1,\ldots ,n_e\rangle$ be a numerical
semigroup minimally generated by $\{n_1,\ldots,n_e\}$ and let $K[H] = K[ t^{n_1}, \ldots , t^{n_e}]$
be the semigroup ring of $H$, where $t$ is a variable and $K$ is any field.
We can represent $K[H]$ as a quotient ring of a polynomial ring $S=K[x_1,\ldots , x_e]$
 as $K[H] = S/I_{H}$, where $I_H$ is the kernel of the $K$-algebra homomorphism which maps
  $x_i\to t^n_i$.  We call $I_H$ the defining ideal of $K[H]$. The ideal $I_H$ is a binomial ideal, whose binomials correspond to   pairs of factorizations of elements of $H$. 

\bigskip

The concept of inverse system introduced by Macaulay is very useful to study 
Artinian Gorenstein rings. If $A = S/ I$ is an Artinian Gorenstein local 
ring which is a quotient of a polynomial ring $S=k[x_1,\ldots , x_e]$ over
a field $k$ and $\sqrt{I} = \nn := (x_1,\ldots , x_e)$, then putting 
$\EE= E_S( S/\nn) \cong k[x_1^{-1},\ldots , x_e^{-1}]$ to be the injective 
envelope of $S/\nn$, then $[0:_{\EE} I]$ is generated by s single element 
$J\in \EE$ since $A$ is Gorenstein and we call $J$ the inverse polynomial 
of $I$. See  \ref{Inv-syst}, \ref{Inv-polyn}.  

The aim of this paper is to analyze the structure of $k[H]$ by the aid of 
the inverse polynomial corresponding to $h \in H$.

In particular, if $H$ is symmetric, since $K[H]/(t^h)$ is 
Artinian Gorenstein ring, $K[H]/(t^h)$ corresponds to an inverse 
polynomial $J_{H,h}\in \EE$. The Frobenius number $\Fr(H)$ of $H$ is 
a very important number which is equal to the $a$ invariant of 
$k[H]$ and in particular, we can hope that 
determining the shape of $J_{\Fr(H)+n_i}$ will be very useful to analyze 
the structure of $k[H]$. 

Also, we can hope also to analyze almost symmetric semigroups,
or, equivalently, the almost Gorenstein numerical semigroup rings
(\cite{BF}) from the inverse polynomial $J_{\Fr(H) + n_i}$ 
(\ref{AS}).   
 
In \S 1, we introduce basic concepts concerning numerical semigroups and 
the semigroup rings and the inverse polynomials.

In \S 2, we introduce the inverse polynomial $J_{H,h}$ attached to 
a numerical semigroup $H$ and element $h\in H$. In Theorem \ref{AS}
we give a characterization of symmetric and almost symmetric semigroups
in terms of $\Ann_S( J_{H, \Fr(H) + n_i})$.  We also give the characterization 
of symmetric numerical semigroups of small multiplicity using inverse polynomial. We will show that the characterization of pseudo-symmetric 
numerical semigroups generated by 3 elements is obtained directly using 
the inverse polynomial. 

In \S 3 we discuss about the inverse polynomial of a semigroup $H$ which is 
obtained by gluing of numerical semigroups $H_1$ and $H_2$. 
Since a numerical semigroups which is a complete intersection is obtained 
by successive gluings, we determine the shape of inverse polynomial 
when $H$ is a complete intersection.

In \S 4 we investigate the case when $\Fr(H) + n_i$ has unique factorization 
(that is, $J_{H, \Fr(H) + n_i}$ is a monomial) and conclude that $H$ is "free" (\ref{monomial}).

In \S 5, we deduce structure theorems for symmetric numerical semigroups 
of small multiplicity by classifying inverse polynomials $J$ with  
small $\ell_S( S/\Ann_S(J))$. 

in \S 6, we give a proof of Bresinsky's Theorem on symmetric numerical 
semigroups generated by $4$ elements which is not a complete intersection 
using the inverse polynomial $J_{H, \Fr(H)+n_i}$.

In \S 7, we calculate the inverse polynomial $J_{H, \Fr(H)+n_i}$ 
in the case $H$ is symmetric and generated by $4$ elements and not a 
complete intersection. 
 
\begin{acknowledgement}
This work began inspired by the talk of M.E. Rossi titled 
"A constructive approach to one-dimensional Gorenstein k-algebras" 
on Dec. 1, 2020 at Virtual Commutative Algebra Seminar at IIT Bombay.
The 2nd author thank M.E. Rossi fro the talk and also the organizers
 of the seminar, especially to Jugal Verma for their efforts to continue
  the Seminar and for inviting the second author to this Seminar.  
\end{acknowledgement}

\section{Basic concepts}
In this section we fix notation and recall the basic
 definitions and concepts which will be used in this paper.

\begin{defn}\label{defk[H]}
 Let $H=\langle n_1,\ldots ,n_e\rangle \subset \NN$ 
be a  numerical semigroup minimally generated by $e$ elements.
We will use the following notations.
\begin{enumerate}
\item Let $k$ be any field and let 
$k[H]= k[t^h\;|\; h\in H]\subset k[t]$ be the semigroup ring of
 $H$ over a field $k$. We put $H_+ = H \setminus \{0\}$ and we denote
  $\m_H$ the unique homogenous maximal ideal of $k[H]$. 
  
\item   Let $S=k[x_1,\ldots,x_e]$ be a 
polynomial ring over $k$ in the indeterminates $x_1,\ldots,x_e$. 
We denote  $\Phi=\Phi_H : S \to k[H]$ the surjective map defined by 
$\Phi_H(x_i) = t^{n_i}$. We always denote $I_H = \Ker( \Phi_H)$ be 
the defining ideal of $k[H]$. We consider $S$ to be a graded ring 
over $k$ by putting $\deg x_i = n_i$ so that $\Phi_H$ is a homomorphism 
of graded rings.  Note that $I_H$ is generated by 
\[\{\prod_{i=1}^ex_i^{\al_i}-\prod_{i=1}^ex_i^{\bl_i}\;|\; 
\sum_{i=1}^e\al_in_i=\sum_{i=1}^e\bl_in_i\}.\] 

\item For an ideal $I$ of a ring $A$, we denote by $\mu(I)$ the
 minimal number of generator systems of $I$.   

\item For any $i$, $1\le i\le e$, we denote $S_i = S/x_iS$.
Then we have $k[H]/(t^{n_i}) \cong S/ (I_H, x_i) \cong S_i/ (I_H S_i)$.
We note that $\mu(I_H) = \mu((I_H + (x_i))/(x_i))$ for every $i$.

\item For $h\in H_+$, we denote 
\[ \Ap(H, h) := \{ a\in H\;|\; a-h\not\in H\}\]
and call it the Ap\'ery set of $H$ with respect to $h$.
It is easy to see that $\sharp(\Ap(H,h)) = h$ and 
$\{t^a\;|\; a\in  \Ap(H,h)\}$ forms a basis of $k[H]/(t^h)$ over $k$.
\end{enumerate}
\end{defn}

\begin{defn}\label{defH} Let $H=\langle n_1,\ldots ,n_e\rangle \subset \NN$ 
be a  numerical semigroup minimally generated by $e$ elements.
\begin{enumerate}
\item For $n, n'\in \Z$, we denote $n\le_H n'$ if $n'-n\in H$. 
\item For $h\in H$, an expression $h = \sum_{i=1}^e a_in_i$ is called 
a {\it factorization} of $h$ in $H$. If $h$ has only one factorization, then we say $h$ has {\it unique factorizaion}.

\item We denote $\Fr(H) = \max \{ n\in \Z\;|\; n\not\in H\}$ 
and call it the {\it Frobenius number of $H$}.
Also we define $\PF(H) = \{ n\in \Z\;|\; n\not\in H\; {\rm and} \;\forall h\in H_+, n+h \in H \}$. 
We call $f \in \PF(H)$ a pseudo Frobenius number. We define 
$\type(H) = \sharp(\PF(H))$. Note that if $x\in \ZZ, x\not\in H$, then 
for some $f\in \PF(H), x\le_H f$. 
\item We define $\al_i$ to be the minimal positive integer such that
\begin{eqnarray}
\label{minimal}
(*) \quad \al_i n_i = \sum_{j=1, j\ne i}^e \al_{ij} n_j.\nonumber
\end{eqnarray}
Thus $f_i = x_i^{\al_i} - \prod_{j=1, j\ne i}^e x_j^{\al_{ij}}$ 
($1\le i\le e$) is a minimal generator of $I_H$. 
\item We define $H$ is {\it symmetric} if $\type(H) =1$, or
$k[H]$ is Gorenstein  (\cite{Ku}).
\item We define $g(H) = | \NN \setminus H|$. Then we see that $H$ is symmetric 
if and only if $2 g(H) = \Fr(H) +1$. In general, we have 
$2 g(H) \ge \Fr(H) + \type(H)$ and we say that $H$ is {\it almost symmetric}
if $2 g(H) = \Fr(H) + \type(H)$. Then $H$ is almost symmetric if and only if 
$k[H]$ is almost Gorenstein (\cite{BF}). 
\item Let $S$ be a polynomial ring as in \ref{defk[H]} and if $I\subset S$
 is a homogeneous ideal of $S$, then we denote
\[\Deg(S/I) = \{ n \;|\; (S/I)_n \ne 0\} \subset H.\]
\end{enumerate}
\end{defn}

\begin{rem}(\cite{GW}) We see that the local cohomology group 
$\rmH^1_{\mm}(k[H])$ is isomorphic to $k[t,t^{-1}]/k[H]$. 
Hence the set $\{ t^n \;|\; n\in \PF(H)\}$ is a basis of 
$\Soc(\rmH^1_{\mm}(k[H]))$, which implies that 
$\type(H)$ coincides with the Cohen-Macaulay type of $k[H]$ and     
$a( k[H]) : = \max\{ n \;|\; (\rmH^1_{\mm}(k[H]))_n \ne 0\}$ 
coincides with $\Fr(H)$. 
\end{rem}

\bigskip

The concept of inverse system is essential in this paper 
(cf. \cite{GW}, (1.2.11).

\begin{defn}\label{Inv-syst} 
Let $S = k[x_1,\ldots , x_e]$ be a polynomial ring over
 a field $k$ and $\mm = (x_1,\ldots , x_e)$ be the unique homogeneous 
maximal ideal. It is well known that the injective envelope 
$\EE = E_S( S/\mm)$ can be represented by the set of \lq\lq inverse polynomials"
$\EE = k[X_1,\ldots , X_e]$, where we put $X_i = x_i^{-1}$ and 
we define the action of $S$ on $\EE$ by
\[(x_1^{a_1}\cdots x_e^{a_e})\cdot(X_1^{b_1}\cdots X_e^{b_e}) = 
\left\{\begin{array}{ll}  X_1^{b_1-a_1}\cdots X_e^{b_e-a_e} & 
(\forall i, b_i\ge a_i) \\
0 & \rm{(otherwise)} 
\end{array} \right.\]
\end{defn}

The following fact is fundamental in our argument.

\begin{prop}\label{Inv-polyn} Let $I$ be an $\mm$ primary ideal of $S$. Then 
$\Ann_{\EE}(I)$ is generated by $\type(S/I)$ elements as an $S$ module.
In particular, if $S/I$ is Gorenstein, then $\Ann_{\EE}(I)$ is generated by 
single element $J_I$. Thus we have a one-to-one correspondence between 
the set of $\mm$ primary irreducible ideals and the set of inverse polynomials
in  $\EE$ (mod multiplication of an element of $k$). \par  
If $I$ is a graded ideal of, then $\deg J_I = a(A)$, the $a$-invariant of 
$A/I$. 

In particular, if $H$ is a numerical semigroup and $h\in H_+$, then 
$a( k[H]/(t^h)) = \Fr(H) + h$ (\cite{GW}, 3.1.6).  
\end{prop}

\section{Inverse polynomial attached to factorizations of $h\in H$ and 
description of $I_H$}

\begin{defn}\label{defx^a}
We use the notations as in \S 1. We denote for $\ab = (a_1, \ldots, a_e) 
 \in \NN^e$,  $x^{\ab} = \prod_{1=1}^e x_i^{a_i} \in S$ and 
 $X^{\ab} = \prod_{1=1}^e X_i^{a_i} \in \EE$.
For $\bb = (b_1,\ldots b_e)$, we denote $\bb \ge \ab$ if $b_i\ge a_i$ for 
every $i$. Thus we have $x^{\ab}X^{\bb} = X^{\bb -\ab}$ if $\bb\ge \ab$ 
and $0$ otherwise.

Also we denote $\deg_H (\ab) = \sum_{i=1}^e a_in_i$ and $\ord(\ab) 
= \sum_{i=1}^e a_i$. For an inverse polynomial
$J = \sum_{i=1}^s X^{\ab_i}$, we define 
\[\ord J = \max \{ \ord(\ab_i) \;| \; 1\le i\le s\}.\] 
\end{defn}

We attach an 
 inverse polynomial $J_{H,h}\in \EE=E_S(S/\m)$ for $h\in H$.    

\begin{defn}\label{J_H,h-def} We define for $h\in H$, 
\[J_{H,h} = \sum_{\ab\in \NN^e, \deg_H{\ab} = h} X^{\ab}.\]
\end{defn}

The defining ideal $I_H$ of $k[H]$ is described by annhilaters of $J_{H,h}$.

\begin{thm}\label{I_H+(x_i)}
Let $H=\langle n_1,\ldots ,n_e\rangle$,
${\ab}\in\NN^e$ and $h=\deg_H{\ab}$. Then
\[
I_H+(x^{\ab})=\bigcap_{f\in\PF(H)} \Ann_S ( J_{H, f+h} ).
\]
Hence for each $n_i$, 
\[I_H + (x_i) = \bigcap_{f\in\PF(H)} \Ann_S ( J_{H, f+ n_i} ). \]
In particular,  if $H$ is symmetric, then 
$I_H + (x_i) = \Ann_S (J_{H, \Fr(H) + n_i})$.
\end{thm}
\begin{rem} If $\ab, \bb \in \NN^e$ with $\deg_H \ab = h = \deg_H \bb$, 
then $x^{\ab} - x^{\bb}\in I_H$ and $I_H+(x^{\ab})=I_H+(x^{\bb})$. 
\end{rem}

To prove the theorem, we note a Lemma.

\begin{lem}\label{mJne0}
Let $H$ be as in Definition \ref{defH}, 
$m\in H$, $\ab\in\NN^e$
and put $h=\deg_H \ab$. Then we have 
\begin{enumerate}
\item
If $h\leq_H m$,
then
$x^{\ab} \cdot J_{H, m}=J_{H, m-h}$.
\item
If $h\not\leq_H m$,
then $x^{\ab}\cdot J_{H, m}=0$, thus $x^{\ab}\in\Ann_S(J_{H, m})$.
\item Hence $\Ann_S(J_{H,m})$ is generated by $I_H$ and 
$\{x^{\ab}\;|\; \deg \ab\not\le_H m\}$ and 
we have 
\[\Deg(S/ \Ann_S(J_{H,m})) = \{ h\in H\;|\; h\le_H m \}.\]

Since  $h \in \Deg ( S/ \Ann_S(J_{H,m}))$ if and only if 
$m-h\in \Deg ( S/ \Ann_S(J_{H,m}))$, 
$\dim_k S/ \Ann_S(J_{H,m})$ is even except the case $m$ is even and 
$m/2\in H$.  
\end{enumerate}
\end{lem}

\begin{proof} Assume $h\leq_H m$.
For $\bb\in\NN^e$ with $\deg_H \bb=m-h$,
the degree of the monomial $x^{\ab+\bb}$ is $m$,
thus $X^{\ab+\bb}$ appears in $J_{H, m}$.
Conversely, 
for $\bb\in\NN^e$ with $\deg_H\bb=m$,
if $x^{\ab}\cdot X^{\bb}\ne0$,
then $\ab\leq \bb$ and 
the degree of the monomial $x^{\bb-\ab}$ is $m-h$.
Thus $X^{\bb-\ab}$ appears in $J_{H, m-h}$.
This proves $(1)$.

If $h\not\le_H m$, then for any $\bb\in \NN^e$ with $\deg_H \bb = m$
 we have $\ab \not\le \bb$ hence $x^{\ab} J_{H,m} = 0$.   
\end{proof}


\begin{proof}[Proof of Theorem \ref{I_H+(x_i)}]
First we show that $I_H + (x^{\ab})\subset \Ann_S(J_{H, f + h})$ for every 
$f\in \PF(H)$. 
Since $h\not\leq_Hf+h$, $x^{\ab} \in \Ann_S(J_{H, f + h})$ 
by Lemma~\ref{mJne0} (2).
Take a generator $x^{\bb} - x^{\cb}$ of $I_H$, 
where $\bb, \cb\in\NN^e$ 
 with $\deg_H\bb = \deg_H\cb=d$. We may assume that 
$d\le_H f + h$, since otherwise 
$x^{\bb}, x^{\cb}\in \Ann_S(J_{H, f+n_i})$ 
by Lemma \ref{mJne0} (2). 
Then, by Lemma~\ref{mJne0} (1),
\[
(x^{\bb} - x^{\cb})J_{H, f+h}=J_{H, f+h-d}-J_{H, f+h-d}=0.
\]
Hence, in any case, $x^{\bb} - x^{\cb}\in\Ann_S(J_{H, f+h})$,
therefore
$I_H + (x^{\ab})\subset \Ann_S(J_{H, f + h})$.

Since we have shown 
$I_H + (x^{\ab})\subset \bigcap_{f\in \PF(H)} \Ann_S(J_{H, f + h})$,
we have a surjection $S/(I_H + (x^{\ab}))\to S/\bigcap_{f\in \PF(H)} 
\Ann_S(J_{H, f + h})$.  
Note that, for $\bb\in\NN^e$, the image of $x^{\bb}$
  in $S/(I_H + (x^{\ab}))$ is not zero
if and only if $\deg_H\bb\notin h+H$, i.e. 
$\deg_H\bb\in\Ap(H, h)$.
Then, there is $f\in\PF(H)$ satisfying
$\deg_H\bb\leq_H f+h$
and hence $x^{\bb}J_{H, f+h}=J_{H, f+h-\deg_H\bb}\ne0$
by Lemma~\ref{mJne0} (1).
This shows that  $S/(I_H + (x^{\ab}))\to S/\bigcap_{f\in \PF(H)} 
J_{H, f + h}$ is also an injection.  
\end{proof}

From Theorem \ref{I_H+(x_i)} we have the following.

\begin{thm}\label{AS} Let  $H=\langle n_1,\ldots ,n_e\rangle$. Then 
for any $h \in H_+$, we have
\[\dim_k (S/ \Ann_S( J_{H, \Fr(H) +h})) \le h - (\type H - 1)\] 
and the equality holds for some $h\in H_+$ if and only if 
$H$ is almost symmetric.
\end{thm}

\begin{proof} 
Since $I_H\subset \Ann_S(J_{H, \Fr(H)+ h})$ by 
\ref{I_H+(x_i)} and since the basis of $k[H]/(t^h)$ is given 
by $\Ap(H,h)$, the basis of $S/\Ann_S(J_{H, \Fr(H) + h})$ is given by 
\[\{h' \in \Ap(H,h)\;|\; h' \le_H \Fr(H) + h\}.\]

Let $\PF(H) = \{ f_1,\ldots , f_t\}$ with $t= \type(H)$ and 
$f_t = \Fr(H)$. We denote also $\PF'(H) = \PF(H) \setminus \{\Fr(H)\}$. 
If $f \in \PF'(H)$, then $f + h \in \Ap(H,h)$ and since 
$f + h \not\le_H \Fr(H) + h$, 
\[\{h' \in \Ap(H,h)\;|\; h' \le_H \Fr(H) + h\}\subset 
\{h'\in \Ap(H,h)\;|\; h'-h\not\in \PF'(H)\}\]
and since the latter set has cardinality $h - (\type(H) -1)$,
we have the inequality 
$\dim_k (S/ \Ann_S( J_{H, \Fr(H) +h})) \le h - (\type H - 1)$.

If equality holds, then since $\Fr(H) - f\not\le_H \Fr(H)$ for any 
$f\in \PF'(H)$, we must have $\Fr(H) - f\in \PF'(H)$ for any 
$f \in \PF'(H)$.  
 
This implies that $H$ is almost symmetric by the result of H. Nari (\cite{N}). 
\end{proof}

We give examples of $J_{H, \Fr(H) +n_i}$ for $e= 2,3$.

\begin{ex}
Let $H=\langle n_1, n_2\rangle$
where $n_1$ and $n_2$ are coprime.
For $h\in H$, 
we write $h=c_1n_1+c_2n_2$ where $c_1, c_2\geq0$ and 
$c_1<n_2$. Then
\[
J_{H, h}=X_1^{c_1}X_2^{c_2}+X_1^{c_1+n_2}X_2^{c_2-n_1}+
\cdots+X_1^{c_1+dn_2}X_2^{c_2-dn_1},
\]
where $c_2-dn_1<n_1$.
Since $\Fr(H)=n_1n_2-n_1-n_2$,
\[
J_{H, \Fr(H)+n_1}=X_2^{n_1-1},
\qquad
J_{H, \Fr(H)+n_2}=X_1^{n_2-1}
\]
are monomials.
\end{ex}

\begin{ex} Assume $e=3$.  
\begin{enumerate}
\item If $H$ is symmetric, then by \cite{W73}, we can write 
$n_1 = ad, n_2 = bd$ and $n_3=c$, where $(a,b), (c,d)$ are coprime and 
$c \in \langle a,b\rangle \setminus \{a,b\}$. 
Then $\Fr(H) = (ab-a-b)d + c(d-1)$ and we have 
$J_{H,\Fr(H)+n_1}=X_2^{a-1}X_3^{d-1}$,
$J_{H,\Fr(H)+n_2}=X_1^{b-1}X_3^{d-1}$,  
$J_{H,\Fr(H)+n_3}=\sum_{ap+bq =c + ab-a-b} X_1^pX_2^q$.
Tuus $J_{H,\Fr(H)+n_i}$ ($i=1,2$) are monomials.
\item  If $H$ is not symmetric, 
then   $I_H$ is 
 generated by the maximal minors of the matrix  
\[\left( 
\begin{array}{lll}
x_1^\alpha   &  x_2^\beta    &      x_3^\gamma  \\
x_2^{\beta'} &  x_3^{\gamma'}&      x_1^{\alpha'}
\end{array}
\right)\]
for some positive integers $\alpha,\beta, \gamma, \alpha ', \beta',   \gamma'$ 
(cf. \cite{He}). As is shown in \cite{NNW}, we have
$n_1 =\beta \gamma +\beta '\gamma +\beta '\gamma ', 
n_2= \gamma \alpha +\gamma ' \alpha+\gamma ' \alpha', 
n_3 = \alpha \beta +\alpha '\beta +\alpha '\beta '$ and 
putting $N= n_1+n_2+n_3$, we have 
\[\PF(H) = \{ f=\alpha n_1+(\gamma +\gamma ')n_3-N, 
f'=\beta 'n_2+(\gamma +\gamma ')n_3-N\}.\]
If $f > f'$, then $f = \Fr(H)$, \par\noindent
$J_{H, \Fr(H) + n_2}= 
X_1^{\alpha-1}X_3^{\gamma+\gamma'-1}$ 
and 
$\dim_k( S/ \Ann_S(J_{H, \Fr(H) + n_2}) = \alpha(\gamma+\gamma') 
\le n_2 - 1 =\alpha (\gamma +\gamma') + \alpha' \gamma' -1$. 
Hence by  
Theorem \ref{AS}, $H$ is almost symmetric if and only if $\alpha' = 
\gamma' = 1$.  Actually, in this case, from 
$\dim_k (S/ \Ann_S(J_{H, \Fr(H) + n_3}) = n_3-1$, we get also $\beta' =1$. 
If $H$ is almost symmetric and $f < f'$, then we have $\alpha = \beta = 
\gamma =1$ likewise.        
Hence we can reprove Corollary 3.3 of \cite{NNW}. 
\end{enumerate}
\end{ex}

It is a very interesting question
to determine whether
there exists a symmetric numerical semigroup 
whose inverse polynomial is a given inverse 
 polynomial in $\EE$.
For this question, 
We give an example and a proposition.

\begin{ex}
Let $J=X_2^{13}+X_3^{11}+X_2X_3X_4^7\in\EE$.
To make $J$ be homogeneous, 
i.e. $\deg X_2^{13}=\deg X_3^{11}=\deg X_2X_3X_4^7$,
we have
\[
\deg X_2=11t,\qquad \deg X_3=13t,\qquad \deg X_4=17t,
\]
where $t>0$ and $t\in \ZZ$. 
 Assume that $J = J_{H, \Fr(H)+n_1}$ for some symmetric 
semigroup $H =\langle n_1, 11t, 13t, 17t\rangle$. 
Then we have $x_3^3 - x_2^2x_4\in I_H$ since $39t = 
3 \cdot 13t = 2 \cdot 11t + 17t$.  But we have 
$(x_3^3 - x_2^2x_4) J = X_3^8 \ne 0$, which  contradicts 
Theorem \ref{I_H+(x_i)}. This contradiction is caused by 
the fact $J$ is not the sum of all inverse monomials of
degree $143t$. \par
If we put $H' =\langle 11, 13, 17\rangle$ and 
\[J'=J_{H',143}=X_2^{13}+X_3^{11}+
X_2^2X_3^8X_3+X_2^4X_3^5X_3^2+X_2^6X_3^2X_3^3+
X_2X_3X_4^7, \]
we have $\dim_k S/\Ann_S(J') = 84$ by \ref{mJne0} (3). 
\par
Now, we show that there is no symmetric smigroup  
$H = \langle n_1, 11t, 13t, 17t \rangle$  with 
$J_{H, \Fr(H)+ n_1} = J'$. If that is the case, 
we must have $n_1=84$ and since $\PF(H') = \{49, 53\}$, 
$84 \in H'$ and by \ref{glu-def}, we have 
$\PF(H) = \{49t+84(t-1), 53t + 84(t-1)\}$, which implies 
$H$ is not symmetric. 
\end{ex}

The following Proposition shows that if the order of monomials appearing 
in $J$ is $2$ except one, then $J = J_{H, \Fr(H) + n_i}$ for some symmetric 
semigroup $H$. 

\begin{prop}\label{type1}
Let $e>1$ and $c>0$.
Assume that either $e$ or $c$ is even.
We define
\[
H_{e, c}
=\begin{cases}
\langle e+1, e+2, e+3,
 \dots, 2e
\rangle\\
\qquad\text{if $c=1$},\\
\langle e+c, e+c+1, \frac{c^2+(e+2)c+4}2,
 \frac{c^2+(e+2)c+6}2, \dots, \frac{c^2+(e+2)c+2(e-1)}2
\rangle\\
\qquad\text{if $c$ is even},\\
\langle e+c, e+c+2, \frac{c^2+(e+3)c+8-e}2,
 \frac{c^2+(e+3)c+12-e}2, \dots, \frac{c^2+(e+3)c+4(e-1)-e}2
\rangle\\
\qquad\text{if $c>1$ is odd and $e$ is even}.
\end{cases}
\]
Then
$H_{e, c}$ is symmetric with multiplicity $e+c$ and
\[
\Fr(H_{e, c})=
\begin{cases}
2e+1
& \text{if $c=1$},\\
c^2+(e+1)c+1
& \text{if $c$ is even},\\
c^2+(e+2)c+2
& \text{if $c>1$ is odd and $e$ is even}.
\end{cases}
\]
Further,
\[
J_{H_{e, c}, (\Fr(H_{e, c})+e+c)}=
\begin{cases}
X_2X_e+X_3X_{e-1}+\cdots+X_{e'}X_{e'+2}+X_{e'+1}^2
& \text{if $c=1$ and $e$ is even},\\
X_2X_e+X_3X_{e-1}+\cdots+X_{e'+1}X_{e'+2}
& \text{if $c=1$ and $e$ is odd},\\
X_2^{c+1}+X_3X_e+\cdots+X_{e'+1}X_{e'+2}
& \text{if $c>1$ and $e$ is even},\\
X_2^{c+1}+X_3X_e+\cdots+X_{e'+1}X_{e'+3}+X_{e'+2}^2
& \text{if $c>1$ and $e$ is odd},
\end{cases}
\]
where $e=2e'+1$ if $e$ is odd and $e=2e'$ if $e$ is even.
\end{prop}

\begin{proof}
If $c=1$, then the assertion follows from direct computation.
From now, we assume $c>1$.
Put 
$n_2=e+c+1$ (resp. $n_2=e+c+2$)
and $n_i=\frac{c^2+(e+2)c+2(i-1)}2$
(resp. $n_i=\frac{c^2+(e+3)c+4(i-1)-e}2$)
for $i>2$, if $c$ is even (resp. $c>1$ is odd and $e$ is even).
We have to check that
$H_{e, c}$ is minimally generated by $n_1, \dots, n_e$.
We have 
\[
n_3-n_1=\frac{c^2+ec+4-2e}2\geq\frac{2^2+2e+4-2e}2=4>0
\qquad\text{if }c>1,
\]
thus $\embdim H_{e, c}=n_1$.
And $n_3\geq n_1+4$ implies
$n_i>i$ (resp. $n_i\geq2i$) 
if $c$ is even
(resp. $c$ is odd and $e$ is even).
Note
\begin{gather*}
(c+1)n_2=n_3+n_e=n_4+n_{e-1}=\cdots\\
\intertext{And, for $i>2$, }
\begin{aligned}n_i&\equiv c+i-1\pmod {n_1} 
&\text{and}&\quad n_i>c+i-1  
\quad\text{ if $c$ is even},\\
n_i&\equiv 2(c+i-1) \pmod {n_1}
&\text{and}&\quad n_i>2(c+i-1)  
\quad\text{ if $c$ is odd and $e$ is even.}
\end{aligned}
\end{gather*}
Thus, if $i>1$, $i'>2$ and $i+i'\leq e$,
then
\[
n_{i}+n_{i'}-n_{i+i'}
=\begin{cases}
n_i-i>0 & \text{if $c$ is even},
\\ n_i-2i>0 & \text{if $c$ is odd and $e$ is even.}
\end{cases}
\]
Further,
if $i+i'>e$,
then $n_i+n_{i'}>n_i+n_{e-i}>n_e>n_{i+i'-e}$,
hence we have
\[
n_i+n_{i'}>n_{i''}
\qquad\text{if }i>1, i'>2\text{ and }i+i''\equiv i''\pmod{n_1}.
\]

If $n_i\notin\Ap(H_{e, c}, n_1)$, 
then we may write
$n_i=\sum_{j<i}a_jn_j$ where $a_j\geq0$ and $c_1\ne0$,
since $n_j>n_i$ if $j>i$.
Note $i\ne2$.
Put $j'=\max\{ j : a_j\ne0\}$.
Then $j'>2$,
otherwise
 $a_2n_2\equiv n_j\pmod{n_1}$
implies
$a_2>c+1$ thus $a_2n_2>(c+1)n_2>n_i$.
Then
$n_i=\sum_{j<i}a_jn_j>n_i$, by the above argument,
a contradiction.
Hence $n_i\in\Ap(H_{e, c}, n_1)$ for each $i$.
Therefore
\[
\Ap(H_{e, c}, n_1)
=\{1, n_2, 2n_2, \dots, (c+1)n_2, n_3, n_4, \dots, n_e\}
\]
and $H_{e, c}$ is minimally generated by $n_1, \dots, n_e$.
Then the assertion is clear.
\end{proof}

\begin{ex}
In the notation as in \ref{type1},  
\begin{align*}
H_{4, 1} &=\langle 5, 6, 7, 8\rangle\\
H_{5, 2} &=\langle 7, 8, 11, 12, 13\rangle\\
H_{6, 3} &=\langle 9, 11, 19, 21, 23\rangle\\
\end{align*}
\end{ex}

\section{Gluing and Inverse Polynomials}

\begin{defn}\label{glu-def}
Let $H_1$, $H_2$ be numerical semigroups, 
$d_1\in H_2$, and $d_2\in H_1$.
Assume that $d_1$ and $d_2$ are coprime.
We say that
\[
H=\langle d_1H_1, d_2H_2 \rangle
=\{ d_1h_1+d_2h_2 : h_1\in H_1, h_2\in H_2\}
\]
is a gluing of $H_1$ and $H_2$.

We always assume that
$d_1n$ (resp. $d_2n$) is a generator of $H$
if $n$ is a generator of $H_1$ (resp. $H_2$).
In particular, $d_1$ (resp. $d_2$) is not a multiple of
a generator of $H_2$ (resp. $H_1$).
\end{defn}


Note that if $k[H_1]\subset k[t_1]$ and 
$k[H_2]\subset k[t_2]$, then
\[k[H] = (k[H_1]\otimes_k k[H_2])/(t_1^{d_2} - t_2^{d_1})\]
by putting $\deg(t_1)=d_1$ and $\deg(t_2) = d_2$ in $k[H]$. 
Hence we have
\[\type(H) = \type(H_1)\type(H_2) \quad {\rm and}\]
\[\PF(H) = \{d_1 f_1 + d_2f_2 + d_1d_2\;|\; f_1\in \PF(H_1), f_2\in 
\PF(H_2)\}\]   
In particular, we have 
\[\Fr(H) = d_1\Fr(H_1) + d_2\Fr(H_2) + d_1d_2.\]

Note that $k[H]$ is Gorenstein (resp. a complete intersection)
 if and only if so are $H_1,H_2$. 

Also note that if $H_2= \NN$, then $\Fr(\NN) = -1$. Hence we have
\[\PF(\langle d_1H, d_2\rangle )= \{ d_1f + d_2(d_1 -1)\;|\; f\in \PF(H_1)\}.\]



\begin{prop}
Let $H_1$, $H_2$ be numerical semigroups, 
$d_1\in H_2$ and $d_2\in H_1$.
Assume that $d_1$ and $d_2$ are coprime.
Put $H=\langle d_1H_1, d_2H_2\rangle$.
Then, for $m=d_1m_1+d_2m_2$ where $m_1\in H_1$ and $m_2\in H_2$,
we have
\begin{align*}
J_{H, m}&=\sum_{d\in\Z}J_{H, d_1(m_1+dd_2)}J_{H, d_2(m_2-dd_1)}\\
&=\sum_{d\in\Z}J_{H_1, m_1+dd_2}J_{H_2, m_2-dd_1}.
\end{align*}
Note that we set $J_{H, n}=0$ if $n\notin H$
and the right hand of the above is finite sum.
\end{prop}

\begin{proof}
Put $e_1=\embdim H_1$ and $e_2=\embdim H_2$
Tnen $\embdim H=e_1+e_2$.
Let $V=\NN^{e_1+e_2}$, $V_1=\NN^{e_1}$
and $V_2=\NN^{e_2}$.
We consider
$V=V_1+V_2$
and if $a\in V_1\subset V$ (resp. $a\in V_2\subset V$),
then $\deg_Ha=d_1\deg_{H_1}a$ (resp. $\deg_Ha=d_2\deg_{H_2}a$).
If $m\in H$,
then there uniquely exist
$m_1'\in\Ap(H_1, d_2)$ and $m_2'\in\Ap(H_2, d_1)$
satisfying
\[
m=d_1m_1'+d_2m_2'+cd_1d_2 (=d_1m_1+d_2m_2).
\]
Choose $a\in V$ with $\deg_Ha=m$.
Then we uniquely write $a=a_1+a_2$ where
$a_1\in V_1$ and $a_2\in V_2$
and there are $b_1\in\Ap(H_1, d_2)$ and $c_1\geq0$
(resp. $b_2\in\Ap(H_2, d_1)$ and $c_2\geq0$)
with $\deg_{H_1}a_1=b_1+c_1d_2$ 
(resp. $\deg_{H_2}a_2=b_2+c_2d_1$). 
Since $m=\deg_Ha=d_1\deg_{H_1}a_1+d_2\deg_{H_2}a_2$,
we have $b_1=m_1'$, $b_2=m_2'$ and $c_1+c_2=c$.
\end{proof}

\begin{cor}\label{glued_monomial}
Let $H=\langle d_1H_1, d_2H_2\rangle$ be
as in Proposition.
And let $f=d_1f_1+d_2f_2+d_1d_2\in\PF(H)$
where $f_1\in\PF(H_1)$ and $f_2\in\PF(H_2)$.
For $h\in\Ap(H_1, d_2)$ (resp. $h\in\Ap(H_2, d_1)$),
\[
J_{H, f+d_1h}
=J_{H_1, f_1+h}J_{H_2, f_2+d_1}
\qquad(resp.\  J_{H, f+d_2h}
=J_{H_1, f_1+d_2}J_{H_2, f_2+h}
).
\]
Further if $H_2=\langle1\rangle$ and $h\in\Ap(H_1, d_2)$,
then
\[
J_{H, f+d_1h}
=J_{H_1, f_1+h}X_{e+1}^{d_1-1},
\qquad
J_{H, f+d_2}
=J_{H_1, f_1+d_2}
\]
where $e=\embdim H_1$ and $I_H\subset k[x_1, \dots, x_e]\otimes k[x_{e+1}]$.
In this case,
$J_{H, f+d_1h}$ is a monomial if and only if
$J_{H_1, f_1+h}$ is a monomial.
\end{cor}

\begin{ex}
Let $H_1=\langle 2, 3\rangle$, $H_2=\langle1\rangle$.
Put $H=\langle 2H_1, 5H_2\rangle=\langle4, 6, 5\rangle$.
Then
\[
J_{H, 7+2\cdot2}=J_{H_1, 1+2} X_3^{2-1}=X_2X_3.
\]
\end{ex}

\begin{ex} (1) Let $H_1 = \langle 3, 4\rangle, H_2 = \langle2, 3\rangle$
and $H = \langle 10, 15, 14, 21\rangle = \langle 5H_1, 7H_2 \rangle$.
Then by \ref{glu-def}, $H$ is a complete intersection and since 
$\Fr(H_1) = 5$ and $\Fr(H_2) = 1$,   
$\Fr(H) = 5\cdot 5 + 7 \cdot 1 + 35 = 67$. We have
\begin{align*}
J_{H, 82}
&=J_{H, 67+5\cdot3}
=J_{H_1, 5+3}J_{H_2, 1+5}
=X_2^2(X_3^3+X_4^2)\\
J_{H, 88}
&=J_{H, 67+7\cdot3}
=J_{H_1, 5+7}J_{H_2, 1+3}
=(X_1^4+X_2^3)X_3^2.
\end{align*}

(2) Let $H_1= \langle 5,6,9\rangle$ and $H = \langle 5H_1, 31\rangle
= \langle 25, 30, 45, 31\rangle$. We have $\Fr(H_1) = 13$.
For simplicity, we write 
$k[H_1]= k[x,y,z]/I_{H_1}, k[H] = k[x,y,z,w]/I_H$ and $X,Y,Z,W$ be 
inverse of $x,y,z,w$. Then we have $\Fr(H) = 5\cdot 13 + (5-1)\cdot 31
= 189$ and  
\begin{gather*}J_{H,214} = J_{H_1,13+5}W^4 = (Y^3+Z^2)W^4, 
\quad J_{219}= J_{H_1,13+6}W^4=X^2ZW^4,\\  
J_{H,234} =  J_{H_1,13+9}W^4= X^2Y^2W^4 \quad {\rm and} \quad
J_{H,220}  = J_{H_1,13+31}=J_{H_1,44} 
\end{gather*}
$J_{H_1,44}$ is a sum of $5$ monomials,
since
\begin{align*}
44&=1\cdot5+2\cdot6+3\cdot9
=1\cdot5+5\cdot6+1\cdot9\\
&=4\cdot5+1\cdot6+2\cdot9
=4\cdot5+4\cdot6+0\cdot9\\
&=7\cdot5+0\cdot6+1\cdot9
\end{align*}
\end{ex}

Assume that $H = \langle d H_1, m\rangle$ with $(d,m)=1$ and $m\not\in H_1$.
Then $H$ is not a gluing but we have the following criterion for $H$ to be  symmetric. 

\begin{prop}\label{symm <dH,m>} 
Let $H_1=\langle n_1,\ldots , n_{e-1}\rangle$ is a numerical 
semigroup with $\emb(H_1) = e-1$ and let $H = \langle d H_1,m \rangle$
with $(d,m)=1$ and $m\not\in H_1$. Then $H$ is symmetric if and only if 
$H' = \langle H_1, m \rangle=\langle n_1,\ldots , n_{e-1},m\rangle$
is symmetric. 
\end{prop}
\begin{proof} Assume $H'$ is symmetric. We will show that 
$\PF(H)= \{d \Fr(H') + (d-1)m\}$. We denote $\phi = d \Fr(H') + (d-1)m$.
Actually, it is easy to show that 
$\phi \not\in H$ and for any $h\in H_+,
h +\phi\in H$. Conversely, take any $x\in \Z, x\not\in H$.
If $x = dx'$, then $x'\not\in H'$ and since $H'$ is symmetric, 
$\Fr(H') - x'\in H_1$ and hence $\phi -x \in H$. Otherwise, take 
$s, 0<s< d$ such that $x\equiv sm\; (\mod d)$. Then we can write 
$x = sm + y$ with $y \in d \Z$ and since $x \not\in H, y/d \not\in H'$.   
This implies that $\phi - x = (d-1-s)m + d(\Fr(H') - y/d) \in H$. \par
Conversely, assume that $H'$ is not symmetric.  Take $f\ne f' \in \PF(H')$ 
and put $\phi = (d-1)m + df, \phi' = (d-1)m + df'$. Then it is easy to see
that $\phi, \phi' \in \PF(H)$, showing that $H$ is not symmetric.      
\end{proof}

\begin{rem} The analogous statement for "almost symmetric" does not hold
(\cite{N}, 6.7).  
\end{rem}

\section{Defining equations of $k[H]$ when $H$ is a
complete intersection.} 

In this section,
we investegate the case when $J_{H, \Fr(H) + n_i}$ is a monomial.
First note
$J_{H, \Fr(H) + n_i}$ is a monomial
if and only if $\Fr(H) + n_i$ has UF,
i.e. a unique factorization in $H$.

We give definition of free numerical semigroup,
which is found in \cite{BC} or \cite{RS}.

\begin{defn}[cf {\cite[\S 9.4]{RS}}]
A numerical semigroup $H=\langle n_1,\ldots ,n_e\rangle$
is free, if, by reordering $n_1, \dots, n_e$,
the following condition is satisfied:

\vspace*{2mm}

$n_i/d_{i+1}\in\langle n_1/d_i, \dots, n_{i-1}/d_i\rangle$
where $d_i=\gcd(n_1, \dots, n_{i-1})$
 for $i=2, \dots, e$.

\vspace*{2mm}

\end{defn}

If a numerical semigroup $H$ 
is free, i.e. satifies the above condition,
then it is completely glued, i.e. $I_H$ is a complete intersection,
thus symmetric and
we can compute its Frobenius number
using
\[
\Fr(H)=(d_2/d_3-1)n_2+(d_3/d_4-1)n_3+\cdots(d_e/d_{e+1}-1)n_e-n_1.
\]
Maybe, this formula
is the reason to study free numerical semigroups in long time.

\begin{thm}\label{monomial} 
Assume that $H=\langle n_1,\ldots ,n_e\rangle$ is symmetric.
Then the following are equivalent:
\begin{enumerate}
\item
$H$ is free.
\item
$J_{\Fr(H)+m}$ is a monomial for some $m\in H\setminus\{0\}$.
\end{enumerate}
\end{thm}

Before proving this, we recall Lemma 2 from \cite{W73}.

\begin{lem}\label{Lem2}
Let $H=\langle n_1,\ldots ,n_e\rangle$ be a  semigroup which is a 
complete intersection, and let $(g_1,\ldots , g_{e-1})$ be a set of 
minimal generators of $I_H$. 
If  $\aa_p$ is an ideal generated by a set of $p$ variables, 
then there exists at most $p-1$ $g_i$'s which belong to $\aa_p$.
\end{lem}

\begin{proof}[Proof of Theorem \ref{monomial}]If $H$ is free,
we may assume that
$H$ satisfies the condition
\begin{quote}
$n_i\in\langle n_1/d_i, \dots, n_{i-1}/d_i\rangle$
where $d_i=\gcd(n_1, \dots, n_{i-1})$
 for $i=2, \dots, e$.
\end{quote}
Then 
$H$ is successively glued
and $J_{H, n_1}$ is a monomial by Corollary~\ref{glued_monomial}.

Conversely, assume that 
$J_{\Fr(H)+m}$ is a monomial where $m\in H\setminus\{0\}$.
We may assume $1\in\supp m$, i.e. $m-n_1\in H$.
Then $J_{\Fr(H)+n_1}$ is a monomial by Lemma~\ref{mJne0}.
Put $J_{\Fr(H)+n_1}=\prod_{i=2}^e X_i^{a_i}$.
Then
\[I_H + (x_1) = 
\Ann_{S}(\prod_{i=2}^e X_i^{a_i})
= (x_1)+(x_i^{a_i+1})_{i=2}^e.\] 
In particular, $I_H$ is a complete 
intersection and  $I_H = (f_2, \ldots , f_e)$, where we put 
\[f_i = x_i^{a_i+1} - m_i,\] 
 where $m_i$  is a monomial of 
$\{x_j\;|\; j\ne i\}$ with $x_1 | m_i$.

Put $\aa_p=(x_i)_{i>e-p}$ for $p=1, \dots, e-1$.
By Lemma \ref{Lem2},
there is $l$ with $f_l\notin\aa_{e-1}$.
After reordering the variables, we may assume
$l=2$. Then $m_2=x_1^{b_{11}}$ for $b_{11}>0$.
Note $a_2+1=n_1/\gcd(n_1, n_2)$.
Similarly,
we may assume $f_l\notin\aa_{e-l+1}$ for $l=3, \dots, e-1$.
Then $m_l=x_1^{b_{l1}}\cdots x_{l-1}^{b_{l(l-1)}}$
where $b_{lj}\geq0$ for each $j$
and $l=3, \dots, e-1$.
Put $d_l=\gcd(n_1, \dots, n_{l-1})$ for $l=3, \dots, e+1$.
Since $(a_l+1)n_l=\sum_{j<l}b_{lj}n_j$,
$d_l/d_{l+1}$ divides $a_l+1$ for $l=3, \dots, e$.
Put $a_l+1=c_ld_l/d_{l+1}$  for $l=3, \dots, e$.
We have
\[
n_1
=(a_2+1)\cdots(a_e+1)
=(n_1/d_3)\cdot(c_3d_3/d_4)\cdots(c_ed_e/d_{e+1})
=c_3\cdots c_en_1
\]
and $c_3=\cdots=c_e=1$.
Hence
$n_l/d_{l+1}\in\langle n_1/d_l, \dots, n_{l-1}/d_l\rangle$
for $l=2, \dots, e$
and $H$ is free.
\end{proof}

\begin{rem}
The implication from $(2)$ to $(1)$ in above is essentially proved in
\cite[Proposition 3.5]{Et}.
\end{rem}

\begin{ex}
Let $H=\langle 4, 6, 5\rangle$.
Then $H$ is free with $F(H)=7$ and
$J_{H, 11}=X_2X_3$, $J_{H, 12}=X_1^3+X_3^2$.
Hence, even if $H$ is free,
$J_{H, \Fr(H)+n_i}$ is not always a monomial.
\end{ex}

\begin{ex}
Let $H=\langle 3, 4, 5\rangle$.
Then $F(H)=2$ and
$J_{H, 5}=X_3$, $J_{H, 6}=X_1^2$ and $J_{H, 7}=X_1X_2$.
Hence, every
$J_{\Fr(H)+n_i}$ is a monomial.
However, $H$ is not symmetric and $H$ is not free.
Hence the assumption of the above theorem is necessary.
\end{ex}

\begin{prop}\label{monomial_strong} 
Let $H=\langle n_1,\ldots ,n_e\rangle$.
Then the following are equivalent:
\begin{enumerate}
\item
$I_H$ is a complete intersection generated by the same degree $d$.
\item
$J_{H, \Fr(H)+n_i}$ is a monomial for every $i$.
\end{enumerate}
In this case, there are 
pairwise coprime positive numbers
$\alpha_1, \dots, \alpha_e>1$ 
with $d=\prod_{i=1}^e\alpha_i$ and 
$n_j=\prod_{i\ne j}^e\alpha_i$ for each $j$.
\end{prop}

\begin{proof}
Assume that $I_H$ is generated by binomials of the same degree $d$.
Then $\alpha_in_i=d$ for $i=1, \dots, e$.
And there are exactly $e$ factorizations of $d$ in $H$.
For, if $d=\sum_ia_in_i$ and if there are $j\ne j'$ with $a_ja_{j'}\ne0$,
then we have 
$(\alpha_j-a_j)n_j=\sum_{i\ne j}a_in_i$
and $\alpha_j-a_j<\alpha_j$,
a contradiction to the definition of $\alpha_j$.
Hence $I_H$ is a complete intersection.
Also,
\[
\Fr(H)+n_j=\sum_{i=1}^e(\alpha_i-1)n_i-(\alpha_j-1)n_j
\]
and $n_j=\prod_{i=1}^e\alpha_i/\alpha_j$.
Note that $\alpha_1, \dots, \alpha_e$ are pairwise coprime.
Conversely,
if $J_{\Fr(H)+n_i}$ is a monomial for every $i$,
we have
\[
I_H+(x_i)=(x_i)+(x_j^{\alpha_j})_{j\ne i}
\]
and $\Fr(H)+n_i=\sum_{j\ne i}(\alpha_j-1)n_j$ for each $i$,
thus
$\Fr(H)+\alpha_in_i=\sum_{j=1}^e(\alpha_j-1)n_j$
and $\alpha_in_i=\alpha_jn_j$ for each $i, j$.
This implies
\[
I_H=(x_2^{\alpha_2}-x_1^{\alpha_1}, \dots, x_e^{\alpha_e}-x_1^{\alpha_1}).
\]
In this case, 
$\alpha_1, \dots, \alpha_e$ are pairwise coprime positive numbers
greater than one.
and  $n_j=\prod_{i=1}^e\alpha_i/\alpha_j$ for $j=1, \dots, e$.
\end{proof}

\begin{rem}
If $\alpha_1, \dots, \alpha_e$ are pairwise coprime positive numbers
greater than one,
and
if we put $n_j=\prod_{i=1}^e\alpha_i/\alpha_j$ for $j=1, \dots, e$,
then $H=\langle n_1, \dots, n_e\rangle$
satisfies the conditions of above theorem.
\end{rem}

\section{Symmetric semigroups with  small multiplicity}

Let $H=\langle n_1,\ldots ,n_e\rangle$ be a {\it symmetric} semigroup 
with $\emb(H) = e$ and $m(H) =n_1$. Note that $m(H) \ge e+1$ if $H$ is 
symmetric.
Then since $\ell_{S}(S/\Ann_S(J_{H,\Fr(H)+n_1})) = n_1$, 
if $n_1$ is small in regards to $e$, the possibility of 
$J_{H,\Fr(H)+n_1}$ is very limited and this fact characterizes $H$. 
We classify all the possibilities of $J_{H,\Fr(H)+n_1}$ for $n_1\le e+3$.
Note that for every $n_i, i\ge 2$, $n_i\le_H \Fr(H) + n_1$ and hence 
$J_{H,\Fr(H)+n_1}$ must contains $X_i$. 

\begin{defn} If $M=\prod_{i=1}^e X_i^{a_i}$ is a monomial in 
$\EE = k[X_1,\ldots, X_e]$, we say 
\[\ord M = \sum_{i=1}^e a_i\]  
and if $J = \sum M_j\in \EE$, then 
\[\ord J = \max \ord M_j,\] 
\end{defn} 

\begin{ex} Let $H$ be as above with $n_1=m(H) <n_2< \ldots <n_e$ 
and we put  $J = J_{H, \Fr(H) + n_1}$. If $n_1\le e+3$, the possibilities 
of $J$ is as follows.
\begin{enumerate} 
\item If $n_1= e+1$ and if $e= 2e' +1$ is odd, then 
$J = X_2X_e + X_3X_{e-1}+ \ldots +X_{e'+1}X_{e'+2}$.
\item If $n_1 = e+1$ and if $e=2e'$ is even, then 
$J = X_2X_e +  \ldots +X_{e'}X_{e'+2}+ X_{e'+1}^2$.
\item If $n_1= e+2$ and $e=2e'+1$ is odd, then 
$J = X_2^3 + X_3X_e+ \ldots + X_{e'+1}X_{e'+3}+ X_{e'+2}^2$.
\item If $n_1= e+2$ and $e=2e'$ is even, then 
$J = X_2^3 + X_3X_e+ \ldots + X_{e'+1}X_{e'+2}$. 
\item If $n_1= e+3$ and $e=2e'+1$ is odd, then 
$J = X_i^2X_j + \sum X_kX_l$ satisfying the conditions;\par
 (a) Every $X_p$ appears exactly once;  and \par
 (b) $J$ is homogeneous of degree $\Fr(H) + e+3$.
\item  If $n_1= e+3$ and $e=2e'$ is even, then $J$ is either 
as in (5) or 
\[(6b) \quad J = X_2^4 + X_3X_e+ \ldots + X_{e'+1}X_{e'+2}.\] 
\end{enumerate}
\end{ex}

\begin{rem}
\begin{enumerate}
\item
By Proposition~\ref{type1},
there exists a numerical semigroup $H$
satisfying the condition $(1)$, $(2)$, $(3)$, $(4)$ or $(6b)$
in the example.
\item
If $e$ is odd, then $J$ as (6b) cannot appear, since 
$X_2^4$ and $X_j^2$ cannot appear simultaneously. 
\item
$H=\langle8, 9, 10, 14, 15\rangle$ 
is an example of $(5)$.
For, $\Fr(H)=21$ and
$9+10\cdot2=14+15$ implies
$J_{H, 28}=X_2X_3^2+X_4X_5$.
\end{enumerate}
\end{rem}

\begin{proof} We put $S_1 = S/(x_1)$ and $\m_1 =(x_2,\ldots ,x_e)$.
We put $Q_J = \Ann_{S_1}(J)$. If $\ord J = s$, then 
$Q_J \supset \m_1^{s+1}$ and $\ell_{S_1}((Q_J \cap \m_1^s)/ \m^{s+1})=1$
since $S_1/Q_J$ is Gorenstein. Thus we have $n_1 \ge e + (s -1)$.
\par
We denote by $\bar{\m}_1$ the maximal ideal of $S_1/Q_J$. 
Then $\bar{\m}_1^s \ne 0$ and $\bar{\m}_1^{s+1}=0$.
If $n_1 = e+2$, then $s=3$ and since 
$\ell_{Q_J}(\bar{\m}_1^2/\bar{\m}_1^3)=1$, $J = X_2^3 + \sum M_j$,
where $\ord {M_j} =2$ for every $j$. Hence we have (3), (4).\par

If $n_1= e+3$, either we have $s=3$ or $s=4$. If $s=3$, then 
$\dim_k (\bar{\m}_1^2/\bar{\m}_1^3)=2$ and since only $2$ monomials of 
order $2$ survives, $M = X_i^2X_j$ is the only possibility of 
monomial in $J$ with order $3$.  \par
If $n_1=e+3$ and $s=4$, then since 
$\dim_k (\bar{\m}_1^2/\bar{\m}_1^3)=\dim_k (\bar{\m}_1^3/\bar{\m}_1^4)=1$
generated by the image of 
$J$ must contain $X_2^4$ and other monomials are order $2$. 
\end{proof}

\begin{cor}\label{small_m(H)} 
Let $H = \langle n_1, n_2, \ldots ,n_e\rangle$ be a symmetric numerical semigroup with $n_1 < n_2 < \ldots < n_e$. 
\begin{enumerate} 
\item (\cite{RS} Proposition 4.10) If $n_1=e+1$, then $n_2+n_e= 
n_3+n_{e-1}= \ldots = \Fr(H) + e+1$.
\item If $n_1= e+2$, then $3n_2 = n_3+n_e = n_4+n_{e-1} = \ldots = 
\Fr(H) + e+2$
\end{enumerate}  
\end{cor}

\begin{ex}  
Consider the case $e=4$ and $H=\langle n_1, n_2, n_3, n_4\rangle$.
Let $S_1 = S/(x_1)=k[x_2,x_3,x_4]$ be a polynomial ring of $3$ variables and 
$J=J_{H, n_1}\in {\bf E}$ which is a sum of monic monomials of $X_2,X_3,X_4$ and we assume 
$x_iJ$  ($i=2,3,4$) are not $0$. Then after a suitable permutation of variables, we have;
\begin{enumerate}
\item 
If $\ell_S( S/ \Ann_S(J)) = 5$
(i.e. $n_1=5$), then
$2n_3=n_2+n_4$
and
$J = X_2X_4 + X_3^2$.
\item If $\ell_S( S/ \Ann_S(J)) = 6$, then $J = X_2^3  + X_3X_4$.
\item If $\ell_S( S/ \Ann_S(J)) = 7$, then $J = X_2^4 +X_3X_4$ or 
$X_2^2X_3 + X_4^2$ (in this case the numbers $n_2, n_3, n_4$ need not be 
an increasing sequence).
\end{enumerate} 
\end{ex}

\section{A proof of Bresinsky's Theorem for symmetric $H$ with $e=4$
using inverse polynomial}

By direct computation, we have

\begin{lem}\label{binomial_J_lemma}
Let $J=X^{\ab}+X^{\bb}\in\EE$ where $\ab=(a_i), \bb=(b_i)\in\NN^e$ and $e>0$.
Put $p_i=\max\{a_i, b_i\}+1$ and $q_i=\min\{a_i, b_i\}+1$ for each $i$.
Then
\[
\Ann_SJ
=(x_i^{p_i})_{i=1, \dots, e}
+(x_i^{q_i}x_j^{q_j})_{\begin{subarray}{c}a_i<b_i \\ a_j>b_j\end{subarray}}
+(x^{\ab}-x^{\bb}).
\]
Hence
\begin{enumerate}
\item if $e = 2 e'$ is even, then $\mu(\Ann_SJ)\leq e+e'^2+1$ and 
\item if $e= 2 e' + 1$ is odd, then $\mu(\Ann_SJ)\leq e+ e'(e'+1)$. 
\end{enumerate}
\end{lem}

\begin{ex}
Let $J=X_1\cdots X_s+X_{s+1}\cdots X_e$.
Then 
\[
\Ann_SJ=(x_1^2, \dots, x_e^2)
+(x_ix_j)_{\begin{subarray}{l}i=s+1, \dots ,e \\ j=1, \dots, s\end{subarray}}
+(x_1\cdots x_s-x_{s+1}\cdots x_e).
\]
If $e=2e'\geq 4$ is even and $s=e'$, then we see $\mu(\Ann_SJ)\leq e+e'^2+1$
and if $e = 2e' + 1$ and $s = e'$ or $e'+1$, then we have
$\mu(\Ann_SJ)\leq e+ e'(e'+1)$
\end{ex}

\begin{thm}\label{binomial_J_thm}
Let $H$ be a numerical semigroup of embedding dimension $e$.
Assume 
$J_{H, \Fr(H)+n_1}=X^a+X^b$ where $\ab=(a_i), \bb=(b_i)\in\NN^e$ 
and
there are numbers $1<s<s'<e$ satisfying
$a_i=b_i$ if $i=1, \dots, s$,
$a_i>b_i$ if $i=s+1, \dots, s'$,
and
$a_i<b_i$ if $i=s'+1, \dots, e$
(note $a_1=b_1=0$).
Then
\[
\mu(\Ann_S(J_{H, \Fr(H)+n_1}/(x_1))\leq e+\frac{(e-s)^2}4.
\]
If $s'=s+1$ and $s'<e-1$, then
\[
\mu(\Ann_S(J_{H, \Fr(H)+n_1}/(x_1))= 2e-s-2.
\]
If $s=e-2$, then $s'=e-1$ and
\[
\mu(\Ann_S(J_{H, \Fr(H)+n_1}/(x_1))=2e-s-3=e-1.
\]
In particular, $e=3$,
then ($s=1$ and ) 
$\Ann_S(J_{H, \Fr(H)+n_1}/(x_1))$ is a complete intersection.
\end{thm}

\begin{proof}
Note that $\Ann_SJ_{H, \Fr(H)+n_1}$ contains $x_1$.
And
\[
\mu(\Ann_SJ_{H, \Fr(H)+n_1}/(x_1))\leq
(e-1)+\frac{(e-1)^2}4+1=e+\frac{(e-1)^2}4.
\]
From Lemma~\ref{binomial_J_lemma}, if $s'=s+1$ and $s'<e-1$,
\begin{multline}
\Ann_SJ_{H, \Fr(H)+n_1}/(x_1)\\
=
(x_2^{a_2+1}, \dots, x_s^{a_s+1}, x_{s+2}^{a_{s+2}+1}, \dots, x_e^{a_e+1}, 
x_{s+1}^{b_{s+1}}x_{s+2}^{a_{s+2}},
\dots, x_{s+1}^{b_{s+1}}x_{e}^{a_{e}},
x^a-x^b)
\notag
\end{multline}
and, if $s=e-2$ and $s'=e-1$,
\[
\Ann_SJ_{H, \Fr(H)+n_1}/(x_1)
=(x_2^{a_2+1}, \dots, x_{e-2}^{a_{e-2}+1}, 
x_{e-1}^{b_{e-1}}x_{e}^{a_{e}},
x^a-x^b)
\]
\end{proof}

In case of $e=4$, 
the following corollary follows from Theorem~\ref{binomial_J_thm}.

\begin{cor}\label{J=m+m'} If $e =4$ and if $J_{H,\Fr(H)+n_i}$ is a sum of $2$ monomials, 
then $\Ann_S(J_{H,\Fr(H)+n_i}/(x_1))$ is generated by $3$ or $5$ elements.
\end{cor}

The following Theorem, combined with Theorem 
\ref{binomial_J_thm} gives a proof of Bresinsky's Theorem 
\cite{Br} that $\mu(I_H)= 3$ or $5$ if $H$ is symmetric and 
$e=4$.

\begin{thm}\label{F(H)+n_i} If 
$H=\langle n_1, n_2, n_3, n_4\rangle$ is symmetric which is not a 
complete intersection, then for some $i$, $\Fr(H) + n_i$ has exactly 
$2$ different factorization and hence $J_{H,\Fr(H)+n_i}$ is a sum of 
$2$ monomials. 
\end{thm}

\begin{proof}

Since $H$ is not a complete intersection,
$\Fr(H)+n_i$ does not have UF for each $i$.
By \cite{Br}, we may assume
$\alpha_in_i\ne\alpha_{i'}n_{i'}$ if $i\ne i'$.
By  \cite[Example 4.3]{Et}, we may assume
\[\alpha_4n_4\in\Ap(S, n_1),\quad
\alpha_{i-1}n_{i-1}\in\Ap(S, n_i)\  \text{for}\  i=2, 3, 4,
\]
and
\begin{align*}
\Fr(H)+n_1&=(\alpha_2-1)n_2+(\alpha_3-1)n_3+a_{14}n_4\\
\Fr(H)+n_2&=a_{21}n_1+(\alpha_3-1)n_3+(\alpha_4-1)n_4\\
\Fr(H)+n_3&=(\alpha_1-1)n_1+a_{32}n_2+(\alpha_4-1)n_4\\
\Fr(H)+n_4&=(\alpha_1-1)n_1+(\alpha_2-1)n_2+a_{43}n_3,
\end{align*}
where $0<a_{14}<\alpha_4$, $0<a_{(i-1)i}<\alpha_i$ for $i>1$
(Note that the above equalities follow from an RF-matrix RF$(S)$ for $\Fr(H)$
in \cite[Example 4.3]{Et}).
Without loss of generality,
we may assume $\alpha_in_i<\alpha_4n_4$ for $i=1, 2, 3$.
Then
\[
(\alpha_2-1)n_2+(\alpha_3-1)n_3<2\alpha_4n_4.
\]
Note $\alpha_in_i\notin\Ap(H, n_l)$ if $l\ne i+1$ for $i=2, 3$,
since 
$\alpha_in_i\ne\alpha_{i'}n_{i'}$ if $i\ne i'$.
Thus 
 $\alpha_in_i\not\leq_H\Fr(H)+n_1$ for $i=2, 3$.
From $\Fr(H)+n_1-a_{14}n_4<2\alpha_4n_4$,
it follows that
$\Fr(H)+n_1$ has exactly two factorizations:
\begin{align*}
\Fr(H)+n_1&=(\alpha_2-1)n_2+(\alpha_3-1)n_3+a_{14}n_4\\
&=a_{32}n_2+(\alpha_3-a_{43})n_3+(\alpha_4+a_{14})n_4.
\end{align*}
\end{proof}

We give another proof of Brsinsky's lemma,
which says $\alpha_in_i\ne\alpha_{i'}n_{i'}$ if $i\ne i'$.
To prove it, we give lemma in \cite{Et}.

\begin{lem}[{\cite[Lemma 4.1]{Et}}]\label{nUF}
Let $H=\langle n_1, n_2, n_3, n_4\rangle$.
Assume that
$\Fr(H)+n_i$ does not have UF for any $i$.
Then
\begin{enumerate}
\item
for each $i$, there is $j$ with $\alpha_jn_j\in\Ap(H, n_i)$,
\item
if $\alpha_jn_j, \alpha_{j'}n_{j'}\in\Ap(H, n_i)$,
then $\alpha_jn_j=\alpha_{j'}n_{j'}$,
\item
for each $i$, 
the number of $j$ satisfying $\alpha_jn_j\in\Ap(H, n_i)$
is at most two,
\item
Assume $\alpha_jn_j\in\Ap(H, n_i)$.
Then there exists a unique factorization of 
$\Fr(H)+n_i=\sum_{l\ne i}\mu_ln_l$
under the conditions $\mu_l\geq0$ for each $l$ and $\mu_j<\alpha_j$.
\item
by suitable change of the order of $n_1, \dots, n_4$,
we have two cases:
\begin{enumerate}
\def\labelenumii{(nUF\arabic{enumii})}
\item
$\alpha_4n_4\in\Ap(H, n_1)$,
$\alpha_in_i\in\Ap(H, n_{i+1})$
for $i=1, 2, 3$,
\item
$\alpha_jn_j\in\Ap(H, n_i)$ 
for $\{i, j\}=\{1, 4\}$ or $\{2, 3\}$.
\end{enumerate}
\end{enumerate}
\end{lem}

\begin{lem}[{\cite[Lemma 2]{Br}}]
If $H=\langle n_1, n_2, n_3, n_4\rangle$ is symmetric and not a complete intersection,
then
$\alpha_in_i\ne\alpha_{i'}n_{i'}$. 
\end{lem}

\begin{proof}
By Lemma~\ref{nUF} (5),
we may assume $\alpha_4n_4\in\Ap(H, n_1)$.
Assume $\alpha_1n_1=\alpha_2n_2$
and $\alpha_3n_3=\alpha_4n_4$.
Since $H$ is symmetric and since $(\alpha_2-1)n_2\in\Ap(H, n_1)$,
we have an factorization
\[
\Fr(H)+n_1=(\alpha_2-1)n_2+\mu_3n_3+\mu_4n_4,
\]
where $\mu_3\geq0$ and $0\leq\mu_4<\alpha_4$
by Lemma~\ref{nUF} (4).
From $\alpha_3n_3=\alpha_4n_4$,
we $\mu_3\geq\alpha_3$.

Put $d_{12}=\gcd(n_1, n_2)$ and $d_{12}=\gcd(n_1, n_2)$.
Suppose $d_{12}d_{34}\notin\NN n_1+\NN n_2$
Since $d_{12}d_{34}\in\Z n_1+\Z n_2$,
there are $\nu_1, \nu_2\in\Z$
satisfying
$d_{12}d_{34}=\nu_1n_1-\nu_2n_2$,
and $0<\nu_i<\alpha_i$ for $i=1, 2$.
If $d_{12}d_{34}\in\NN n_3+\NN n_4$,
then
$\nu_1n_1\in\NN n_2+\NN n_3+\NN n_4$,
this cotradicts to $\nu_1<\alpha_1$.
If $d_{12}d_{34}\notin\NN n_3+\NN n_4$,
then
we also have an expression
$d_{12}d_{34}=\nu_3n_3-\nu_4n_4$
with $0<\nu_i<\alpha_i$ for $i=3, 4$.
Then $\nu_2n_2+\nu_3n_3\notin\Ap(H, n_1)$
and this contradicts to $(\alpha_2-1)n_2+\alpha_3n_3\in\Ap(H, n_1)$.
Therefore
we obtain 
$d_{12}d_{34}\in\NN n_1+\NN n_2$
Similarly, 
$d_{12}d_{34}\in\NN n_3+\NN n_4$
Hence $H$ is glued by $\langle n_1/d_{12}, n_2/d_{12}\rangle$
and $\langle n_3/d_{34}, n_4/d_{34}\rangle$,
hence $I(H)$ is a complete intersection.

Assume $\alpha_1n_1\ne\alpha_2n_2$
and $\alpha_3n_3=\alpha_4n_4$.
Then we may assume
$\alpha_2n_2\in\Ap(H, n_3)$ 
and
$\alpha_1n_1\in\Ap(H, n_4)$. 
Thus
$\alpha_1n_1\notin\Ap(H, n_3)$ 
and
$\alpha_2n_2\notin\Ap(H, n_4)$.
By Lemma~\ref{nUF} (4),
\begin{align*}
\Fr(H)+n_3&=(\alpha_1-1)n_1+\mu_2n_2+(\alpha_4-1)n_4\\
\Fr(H)+n_4&=\mu_1n_1+(\alpha_2-1)n_2+(\alpha_3-1)n_3,
\end{align*}
where $0\leq\mu_i<\alpha_i$ for $i=1, 2$.
Then
\[
(\alpha_1-1-\mu_1)n_1=(\alpha_2-1-\mu_2)n_2
\]
and $\mu_i=\alpha_i-1$ for $i=1, 2$.
We also have factorizations
\[
\alpha_1n_1=\nu_2n_2+\nu_3n_3, \qquad
\alpha_2n_2=\nu_1n_1+\nu_4n_4,
\]
where
$0<\nu_i<\alpha_i$ for each $i$
and
\[
(\alpha_1-\nu_1)n_1+(\alpha_2-\nu_2)n_2=\nu_3n_3+\nu_4n_4
\notin\Ap(H, n_i)
\]
for $i=3, 4$.
Hence
\[
(\alpha_1-1)n_1+(\alpha_2-1)n_2\notin\Ap(H, n_3)
\]
and this cotradicts
to $(\alpha_1-1)n_1+(\alpha_2-1)n_2\leq_H\Fr(H)+n_3\in\Ap(H, n_3)$.
Therefore, this case is impossible.
By Lemma~\ref{nUF} (3),
we conclude $\alpha_in_i\ne\alpha_{i'}n_{i'}$ if $i\ne i'$.
\end{proof}

\section{Inverse polynomial of symmetric semigroups generated by $4$ elements.}

In this section, let $H=\langle n_1,\ldots ,n_4\rangle$ be a 
{\bf symmetric} numerical semigroup generated by $4$ elements, 
{\bf which is not a complete intersection.}
We will determine the shape of $J_{H, \Fr(H)+n_i}$ for such $H$.
For that purpose, we recall the structure of such $H$ and $k[H]$.

\begin{thm}\label{4Gor}(\cite{Br}, \cite{BFS}, \cite{W20}) 
If $H=\langle n_1,\ldots ,n_4\rangle$ is symmetric and not a complete intersection, then $I_H = (f_1,\ldots , f_5)$ and we can choose a 
minimal free resolution of $k[H]$ over $S$ 
\[F_{\bullet} =[\; 0\to F_3 \to F_2 \overset{d_2}{\to} F_1 \overset{d_1}{\to} F_0 =S\to k[H] \to0\; ]\] 
so that we have the following properties;
\begin{enumerate}
\item The map $d_2$ is given by the following skew-symmetric matrix;
\[ \bbM = \left(\begin{array}{ccccc} 0 & -x_3^{\al_{43}} & 0 & -x_2^{\al_{32}} & - x_4^{\al_{24}}\\
x_3^{\al_{43}} & 0 & x_4^{\al_{14}} & 0 & -x_1^{\al_{31}}\\
0 &  -x_4^{\al_{14}} & 0 & -x_1^{\al_{21}} & - x_2^{\al_{42}}\\
x_2^{\al_{32}} & 0 & x_1^{\al_{21}} & 0 & -x_3^{\al_{13}} \\
x_4^{\al_{24}} & x_1^{\al_{31}} & x_2^{\al_{42}} & x_3^{\al_{13}} & 0
\end{array}\right).
   \]
\item If $\al_i$ ($1\le i \le 4$) is as is defined in \ref{defH} (4), 
then we have 
\[\al_1=\al_{21}+\al_{31}, \al_2=\al_{32}+\al_{42},
\al_3=\al_{13}+\al_{43}, \al_4=\al_{24}+\al_{14},\]
where every $\al_{ij}$ is a positive integer.
\item $f_i$ is the Paffian of the skew-symmetric $4\times 4$ matrix $\bbM(i)$
obtained by deleting $i$-th row and $i$-th column of $\bbM$.Hence we have;
\[f_1=x_1^{\al_1}-x_3^{\al_{13}}x_4^{\al_{14}}, 
f_2= x_2^{\al_2}-x_4^{\al_{24}}x_1^{\al_{21}},
f_3= x_3^{\al_3}-x_1^{\al_{31}}x_2^{\al_{32}},\]
\[f_4= x_4^{\al_4}-x_2^{\al_{42}}x_3^{\al_{43}},
f_5= x_1^{\al_{21}}x_3^{\al_{43}}-x_2^{\al_{32}}x_4^{\al_{14}}.\]
\item Since $\dim_k S/(I_H + (x_i)) = n_i$ for $1\le i \le 4$, we have 
\[
\] 

\item Putting $N = \sum_{i=1}^4 n_i$, since $F_3\cong S( - \Fr(H) - N)$,
if $(i,j)$ entry $\bbM_{i,j}$ is not $0$, then we have
\[\Fr(H) + N = \deg \bbM_{i,j} + \deg f_i + \deg f_j.\]   
\item If we choose expressions of $\Fr(H) + N$ which does not contain,   
say, $n_1$ from (5), putting $(i,j)=(1,2)$ (resp. $(1,4)$), 
we have 
\[\begin{array}{lcl}
\Fr(H) + n_1 &= & (\al_2-1)n_2 + (\al_3-1)n_3 + (\al_{14}-1)n_4\\
 & = &  (\al_{32}-1)n_2 + (\al_{13}-1)n_3 + (\al_4+\al_{14}-1)n_4.
 \end{array}\]
\item If the monomial $X_2^{\al_{42}}X_3^{\al_{43}}$ does not divide 
$X_2^{\al_{32}-1}X_3^{\al_{13}-1}$, then we have 
\[\begin{array}{lcl}
J_{\Fr(H) + n_1} &= &X_2^{\al_2-1}X_3^{\al_3-1}X_4^{\al_{14}-1}
+ X_2^{\al_{32}-1}X_3^{\al_{13}-1}X_4^{\al_4+\al_{14}-1}\\
& = &X_2^{\al_{32}-1}X_3^{\al_{13}-1}X_4^{\al_{14}-1}(X_4^{\al_4}+
X_2^{\al_{42}}X_3^{\al_{43}}).
\end{array} \]
Otherwise we have 
\[J_{\Fr(H) + n_1} = \sum_{k \ge 0, \al_{32}-1-k\al_{42}\ge 0, 
\al_{13}-1-k\al_{43}\ge 0}X_2^{\al_{32}-1-k\al_{42}}X_3^{\al_{13}-1-k\al_{43}}
X_4^{(k+1)\al_4}.  \]

In fact, if $\Fr(H) + n_1= \sum_{i=2}^4 \beta_in_i$ is another 
factorization of $\Fr(H) +n_1$ other than the ones shown in (6), 
we must have $x_2^{\al_2-1}x_3^{\al_3-1}x_4^{\al_{14}-1} 
- x_2^{\beta_2}x_3^{\beta_3}x_4^{\beta_4}\in I_H$ and hence 
$x_2^{\al_2-1}x_3^{\al_3-1}x_4^{\al_{14}-1} 
- x_2^{\beta_2}x_3^{\beta_3}x_4^{\beta_4}$ must be divisible by 
$x_4^{\al_4} - x_2^{\al_{42}}x_3^{\al_{43}}$.
\end{enumerate}
\end{thm}
\begin{ex}\label{4Gor-ex} (1) If $H = \left<n_1,n_2,n_3,n_4 \right>
=\left<41,99,70,53 \right>$, then $\Fr(H) = 1019$ and
$f_1 = x_1^3-x_3x_4, f_2= x_2^{11}- x_1^2x_4^{19}, f_3= x_3^2 - x_1x_2,
f_4= x_4^{20} - x_2^{10}x_3, f_5 = x_1^2x_3-x_2x_4$ and we have 
\[\al_1=3, \al_{21}= 2, \al_{31}= 1; \al_2= 11, \al_{32}= 1, \al_{42}=10,\]
\[\al_3= 2, \al_{13}=1, \al_{43}=1; \al_4 = 20, \al_{24}= 19, \al_{14}=1.\]
Thus we have from \ref{4Gor} (5), 
\[\Fr(H) + n_1 = 10n_2+n_3 = 20n_4. \] 
and since these expressions are only representations of $\Fr(H) + n_1$, 
we have $J_{H,\Fr(H)+n_1} = X_4^{20} + X_2^{10}X_3$.

(2) If $H = \left<n_1,n_2,n_3,n_4 \right>
=\left<43, 20, 27, 37 \right>$, then $\Fr(H) = 179$ and
$f_1 = x_1^4-x_5x_4, f_2= x_2^4- x_1x_4, f_3= x_3^7 - x_1^3x_2^3,
f_4= x_4^2 - x_2x_3^2, f_5 = x_1x_3^2-x_2^3x_4$ and we have 
\[\al_1=4, \al_{21}= 1, \al_{31}= 3; \al_2= 4, \al_{32}= 3, \al_{42}=1,\]
\[\al_3= 7, \al_{13}=5, \al_{43}=2; \al_4 = 2, \al_{24}= 1, \al_{14}=1.\]
Thus we have from \ref{4Gor} (5), 
\[\Fr(H) + n_1 = 3n_2+6n_3 = 2n_2+ 4n_3+ 2n_4. \]
These factorizations give inverse polynomial
\[ X_2^2X_3^4( X_4^2 + X_2X_3^2).\]  
Since $X_2X_3^2$ divides $ X_2^2X_3^4$, we have

\[J_{H, \Fr(H) + n_1} = \sum_{k =0}^3 (X_2X_3^2)^kX_4^{2(3-k)}.\]

If we put $J = X_2^2X_3^4( X_4^2 + X_2X_3^2)$, $\Ann_S(J)$ is 
$I_H + (x_4^3)$ and $\dim_k(S/ \Ann_S(J)) = 31$.


\end{ex}


\end{document}